\documentclass[11pt]{amsart}

\usepackage{graphicx}
\usepackage{epsfig}
\usepackage{color}
\definecolor{r}{rgb}{0.9,0.3,0.1}
\definecolor{b}{rgb}{0.1,0.3,0.9}

\usepackage{amsmath}
\usepackage{amssymb}
\usepackage{amsthm}
\usepackage{enumerate}
\usepackage{amsbsy}
\usepackage{amsfonts}
\newtheorem{theo}{Theorem}[section]
\newtheorem{defin}[theo]{Definition}
\newtheorem{prop}[theo]{Proposition}
\newtheorem{coro}[theo]{Corollary}
\newtheorem{lemm}[theo]{Lemma}
\newtheorem{rem}[theo]{Remark}

\newcommand{\al}{\alpha}
\newcommand{\be}{\beta}

\newcommand{\Ga}{\Gamma}
\newcommand{\la}{\lambda}

\newcommand{\Om}{\Omega}

\newcommand{\ep}{\epsilon }
\newcommand{\te}{\theta}
\newcommand{\De}{\Delta}
\newcommand{\de}{\delta}

\newcommand{\pa}{\partial}

\newcommand{\R}{{\bf R}^n}

\newcommand{\ri}{\rightarrow}

\begin{document}
\baselineskip=18pt

\title[Parabolic function spaces]
{Properties of parabolic Sobolev and  parabolic Besov spaces}

\author{TongKeun Chang}
\address{Department of Mathematics, yonsei University, Seoul, South Korea}
\email{chang7357@yonsei.ac.kr}

\thanks{The author was supported by the
National Research Foundation of Korea NRF-2010-0016699 }

\begin{abstract}
In this paper, we  characterize   parabolic Besov and parabolic
Sobolev spaces in ${\bf R}^{n+1}$ and ${\bf R}^{n+1}_T, \,\, T
> 0$. We also, study the relation between parabolic Besov spaces in ${\bf R}^{n}_T, \,\, T > 0$ and standard
Besov space in $\R$.
\end{abstract}

\maketitle


\section{Introduction}
\setcounter{equation}{0}

In this paper, we study the properties of  the parabolic Besov
spaces ${\mathcal B}^{\al, \frac12 \al}_p ({\bf R}^{n+1})$ and the
parabolic Sobolev spaces ${\mathcal L}^p_\al({\bf R}^{n+1})$ for $1
\leq p \leq \infty, \,\, \al \in {\bf R}$. We also, study the
relation between the parabolic Besov spaces in ${\bf R}^{n}_T =
\{(X,t) \,| \, X \in {\bf R}^n, \,\, 0 < t < T \}, \,\, 0 < T \leq
\infty $ and the standard Besov space in $\R$.

The parabolic Sobolev spaces were studied in \cite{FT} and \cite{R}.
The authors in \cite{FT} proved the trace theorem of parabolic
functions which is similar to the usual Sobolev spaces (see
\cite{BL}). In fact, the parabolic Besov and the parabolic Sobolev
spaces are particular case of the anisotropic Sobolev spaces and
anisotropic Besov spaces, respectively, with dilation matrix
$\de_\ep =(\ep^2, \ep, \cdots, \ep)$ (see \cite{DT1}, \cite{DT2},
\cite{L} and \cite{N}).

For the properties about the usual Besov and Sobolev spaces, we
refer \cite{BL}, \cite{ N}, \cite{P}, \cite{S}, \cite{ Tr} and the
references therein.

The Besov and Sobolev spaces have being used in boundary value
problems of several elliptic type partial differential equations in
bounded domain in $\R$. When boundary data is given with the
function in some Besov or  Sobolev spaces, one can find the solutions
of the boundary value problems of elliptic type partial differential
equations which contained in the corresponding spaces with boundary
data (see \cite{BS}, \cite{CC}, \cite{ FMM}, \cite{JK}).

Like the  Besov and Sobolev spaces,  functions in parabolic Besov
and Sobolev spaces can be used with  boundary data and solutions of
initial-boundary value problems of parabolic type partial
differential equations in bounded cylinder (see \cite{JM} for the
case heat equation).

In section \ref{sec1},
 we introduce a parabolic Sobolev space
${\mathcal L}^p_\al ({\bf R}^{n+1})$ and parabolic Besov space
${\mathcal B}_p^{\al, \frac12 \al} ({\bf R}^{n+1})$. The properties to
 the parabolic Sobolev and parabolic Besov spaces are also stated.

In section \ref{sec1-2}, we show that
 $f \in {\mathcal L}^p_\al ({\bf R}^{n+1}), \,\, 1< p<\infty, \,\, \al \in {\bf R}$ is equivalent to  $f, D_{X_k} f, D_t^{\frac12} f \in
{\mathcal L}^p_{\al -1} ({\bf R}^{n+1})$ for all $1 \leq k \leq n$ (see Theorem \ref{iterates}), and that $f \in {\mathcal B}^{\al, \frac12 \al}_p
({\bf R}^{n+1})$,   $1 \leq p \leq \infty, \,\, \al \in {\bf R}$ is equivalent to
 $f,
D_{X_k} f, D_t^{\frac12} f \in {\mathcal B}_p^{\al -1, \frac12 \al -\frac12} ({\bf
R}^{n+1})$ for all $1 \leq i \leq n$ (see Theorem \ref{iterateb}).
Here $D^\frac12_t $ is a fractional differential operator whose
Fourier transform in term of $t$ variable  is defined by
$\widehat{D^\frac12_t} f(X,\tau) = |\tau|^\frac12
\widehat{f}(X,\tau)$.

Our result in section \ref{sec1-2} can be compared with  the
results of V. Gopala Rao and B. Frank Jones. In \cite{R},  V.
Gopala Rao showed that $f \in {\mathcal L}_\al^p ({\bf R}^{n+1})$
is equivalent to $ f , \, \, D_t(f* h_1) \in {\mathcal
L}_{\al-1}^p ({\bf R}^{n+1})$, where $*$ is a convolution in ${\bf
R}^{n+1}$ and $h_1(X,t) = c_1 t^{\frac{-n-1}{2}}
e^{-\frac{|X|^2}{4t}}$ if $t > 0$ and $h_1(X,t) =0$ if $t < 0$. In
\cite{J}, B. Frank Jones induced several equivalent norms of
parabolic Besov spaces. He also showed that
$f\in {\mathcal B}^{\al, \frac12 \al}_p ({\bf R}^{n+1}), \,\, 1 \leq p \leq \infty, \,\, \al \in {\bf R}$ is equivalent to $f, \,\, D_{X_k}f \in
{\mathcal B}^{\al-1, \frac12 \al -\frac12}_p ({\bf R}^{n+1})$ and $D_t f \in {\mathcal
B}^{\al -2,\frac12 \al -1}_p ({\bf R}^{n+1})$.

In  section \ref{sec4}, we characterize the parabolic Besov spaces
in ${\bf R}^{n}_T = \{ (X,t) \in {\bf R}^{n+1} \, | \, 0 < t < T \}, \,\, 0 < T \leq \infty$.
We show that the parabolic Besov spaces in
${\bf R}^{n}_T$ are also interpolation spaces and have the same
properties as the Theorem \ref{iterateb}.

In the section \ref{sec5},  we  study the properties of the solution
$u$ of the heat equation with initial data $f\in {\mathcal B}_p^{\al
-\frac2p} ({\bf R}^{n})$. We show that $u \in {\mathcal
B}_p^{\al,\frac12 \al} ({\bf R}^{n}_T)$, and we investigate an
equivalent  relation between the parabolic Besov norm $\|
u\|_{{\mathcal B}_p^{\al,\frac12 \al} ({\bf R}^{n}_T)}$  and the
usual Besov norm $ \| f\|_{{\mathcal B}_p^{\al -\frac2p} ({\bf
R}^{n})}$.
 For $ 0 \leq  \al, \,\, 1 \leq p \leq \infty$ and $ f \in
{\mathcal B}^{\al -\frac2p}_p (\R) $, we define a function by
\begin{align}\label{main4}
 u (X,t) = <f, \Ga(X- \cdot, t)>: =  \left\{\begin{array}{l}  <f, \Ga(X- \cdot, t)>_\al, \quad
  0 \leq \al < \frac2p, \\
\int_{\R} \Ga(X-Y,t) f(Y) dY, \quad \frac2p \leq \al,
\end{array}
\right.
\end{align}
where $\Ga (X,t) = c_n t^{-\frac{n}2} e^{-\frac{|X|^2}{4t}}$ if $t
> 0$ and $\Ga(X,t) =0$ if $t < 0$,  and $<\cdot, \cdot>_\al$ is duality pairing between ${\mathcal B}^{\al -\frac2p}_p (\R)  $
and ${\mathcal B}^{-\al +\frac2p}_q (\R), \,\, \frac1p + \frac1q =1$.  It is  easy to see that
$u$ is a solution to the  heat equation in ${\bf R}^{n}_\infty$ with
the initial value $f $. Our main result in section \ref{sec5} are
stated as follows.
\begin{theo}\label{mainresult}
 Let $f \in {\mathcal
B}^{\al - \frac2p}_p (\R)$ and $u$ be defined by (\ref{main4}).
Let $1 \leq p \leq \infty$,  $\al>0$ and $T < \infty$.
 Then $u\in {\mathcal B}^{\al, \frac12 \al}_p ({\bf R}^{n}_T) $ with
\begin{align*}
\| u \|_{{\mathcal B}^{\al,\frac12 \al}_p ({\bf R}^{n}_T)}  \approx  \| f\|_{
{\mathcal B}^{\al - \frac2p}_p (\R)}.
\end{align*}
\end{theo}
The notation $A \approx  B$  means that there are positive constants
$c$ and $C$ independent of $f$ such that $c \leq \frac{A}{B} \leq
C$. Our result can be compared with the result of H. Triebel. In
section 1.8.1 in  \cite{Tr3}, H. Triebel showed that  for $1 < p <
\infty$ and $\al > \frac2p$ and $m > \frac12(\al -\frac2p)$,
\begin{align*}
 \|f  \|^p_{L^p(\R)} + \int_0^\infty
\int_{\R} t^{mp-\frac12p (\al -\frac2p)} | D_t^mu(X,t)|^p dXdt\approx  \|
f\|^p_{{\mathcal B}^{\al - \frac2p}_p (\R)}.
\end{align*}

In this paper, we denote that  $A \lesssim B$ means  that $A \leq c
B$ for positive constant $c$ depending only on $n, p,$ and $T$. We
denote $\hat{\cdot}$ as the Fourier transform in ${\bf R}, \,\, {\bf
R}^n$ or ${\bf R}^{n+1}$.

\section{Parabolic Sobolev and parabolic Besov spaces on ${\bf R}^{n+1}$ }
\setcounter{equation}{0} \label{sec1}
For $\al \in {\bf R}$, we
consider a distribution $H_\al(\xi,\tau)$  whose Fourier transform
in ${\bf R}^{n+1}$ is defined by
\begin{eqnarray*}
\widehat{ H_{\al}} (\xi,\tau) = c_\al(1 + 4\pi^2 |\xi|^2 + i
\tau)^{-\frac{\al}{2}}, \quad (\xi, \tau) \in {\bf R}^n \times {\bf R}.
\end{eqnarray*}
For $\al \in {\bf R}, \,\, 1\leq p \leq \infty$, we define the
parabolic Sobolev space ${\mathcal L}^p_{\al} ({\bf R}^{n+1})$ by
\begin{eqnarray*}
{\mathcal L}^p_\al ({\bf R}^{n+1}) = \{ f \in {\mathcal S}'({\bf
R}^{n+1}) \, | \, f  = H_{\al} * g, \quad \mbox{for some} \quad g   \in L^p ({\bf R}^{n+1}) \}
\end{eqnarray*}
with norm
\begin{align*}
\|f\|_{{\mathcal L}^p_\al ({\bf R}^{n+1})} : = \| g \|_{L^p({\bf R}^{n+1})} \, ( = \|  H_{-\al} * f  \|_{L^p({\bf R}^{n+1})} ),
\end{align*}
 where
$*$ is a convolution in ${\bf R}^{n+1}$ and ${\mathcal S}^{'}({\bf
R}^{n+1})$
 is the dual space of the Schwartz space
${\mathcal S}({\bf R}^{n+1})$. In particular, when $\al =0$, we have that ${\mathcal L}^p_0({\bf R}^{n+1}) = L^p({\bf R}^{n+1})$.

Next, we define a parabolic Besov space. Let $\phi \in {\mathcal
S} ({\bf R}^{n+1})$ such that
\begin{eqnarray*}
 \left\{\begin{array}{rl}
&\hat \phi(\xi,\tau) > 0  \quad \mbox{ on }  2^{-1} < |\xi| + |\tau|^\frac12 < 2 ,\\
& \hat \phi(\xi,\tau)=0 \quad \mbox{ elsewhere }, \\
&\sum_{ -\infty < i < \infty } \hat \phi(2^{-i}\xi, 2^{-2i}\tau) =1 \ (
(\xi,\tau) \neq (0,0)).
\end{array}\right.
\end{eqnarray*}
We define functions  $\phi_i, \,\, \psi \in {\mathcal S}({\bf R}^{n+1})$ whose Fourier transforms are
 written by
\begin{eqnarray}\label{psi1}
\begin{array}{ll}
\widehat{\phi_i}(\xi, \tau) &= \hat \phi(2^{-i} \xi, 2^{-2i} \tau) \quad (i = 0, \pm 1, \pm 2 , \cdots)\\
\widehat{\psi}(\xi, \tau) & = 1- \sum_{i=1}^\infty \hat \phi (2^{-i}
\xi, 2^{-2i} \tau).
\end{array}
\end{eqnarray}
Note that $\phi_i = 2^{(i+2)n} \phi(2^i X, 2^{2i} t)$.
For $\al \in {\bf R}$ we define the parabolic Besov space
${\mathcal B}^{\al,\frac12 \al}_{pq} ({\bf R}^{n+1})$ by
\begin{eqnarray*}
{\mathcal B}^{\al,\frac12 \al}_{pq} ({\bf R}^{n+1}) = \{ f \in {\mathcal
S}^{'}({\bf R}^{n+1}) \, | \, \|f\|_{{\mathcal B}^{\al, \frac12 \al}_{pq}} <
\infty \, \}
\end{eqnarray*}
with the norms
\begin{align*}
 \|f\|_{{\mathcal B}^{\al, \frac12 \al}_{pq}} :&  = \| \psi * f\|_{L^p} + (
\sum_{ 1 \leq i  < \infty} (2^{\al i} \|\phi_i *
f\|_{L^p})^q)^{\frac1q}, \quad 1 \leq q < \infty,\\
 \|f\|_{{\mathcal B}^{\al, \frac12 \al}_{p\infty}}  :& =  \sup (\| \psi * f\|_{L^p} , \,\,  2^{\al i} \|\phi_i *
f\|_{L^p}),
\end{align*} where $*$ is a convolution in ${\bf
R}^{n+1}$. When $p=q$, we simply denote ${\mathcal B}^{\al,\frac12 \al}_{pp} $
by ${\mathcal B}^{\al,\frac12 \al}_p$.

The following properties can be shown by the same argument as for
the usual Sobolev space and the Besov space in $\R$.
\begin{prop}
\label{prop2}
\begin{itemize}
\item[(1)]
 The definition of $ {\mathcal B}^{\al, \frac12 \al}_{pq} ({\bf
R}^{n+1})$ does not depend on the choice of the function $\phi$,
\item[(2)]
 The real interpolation method gives
\begin{eqnarray*}
( {\mathcal L}^{p}_{\al_0} , {\mathcal L}^{p}_{\al_1} )_{\te, q} =
{\mathcal B}^{\al,\frac12 \al}_{pq}
\end{eqnarray*}
for $ 1 \leq p \leq \infty, \, \, \al = (1-\te) \al_0 + \te \al_1,
0 < \te < 1,$ and
\begin{eqnarray*}
({\mathcal B}_{pq_0}^{\al_0,\frac12 \al_0} , {\mathcal B}_{pq_1}^{\al_1,\frac12 \al_1} )_{\te,
r} = {\mathcal B}^{\al,\frac12 \al}_{pr}
\end{eqnarray*}
for $\al_0 \neq \al_1, \, \, 1 \leq p, r, q_0, q_1  \leq \infty ,
\al = (1-\te) \al_0 + \te \al_1$.

\item[(3)]
For $0 < \al < 2$,
the the parabolic Besov norm $\| f\|_{{\mathcal B}^{\al,\frac12 \al}_p}$ is
equivalent to the norm
\begin{align}\label{besovnorm1}
\|f\|_{L^p} & + \Big(\int_{\R}  \int_{{\bf R} \times {\bf R}}
\frac{|f(X,t) - f(X,s)|^p}{|t-s|^{1 + \frac12
p\al}}dtds dX \Big)^{\frac1p} \\
& + \Big(\int_{{\bf R}}  \int_{\R \times \R} \frac{|f(X+Y,t) -2
f(X,t) - f(X-Y,t)|^p}{|Y|^{n + p\al}}dYdX
  dt\Big)^{\frac1p}\nonumber
\end{align}
if $1 \leq p < \infty$;
\begin{align}\label{besovnorm2}
\|f\|_{L^\infty} & + \sup_{X ,t,s, t \neq s} \frac{|f(X,t) -
f(X,s)|}{|t-s|^{ \frac12
p\al}} \\
& + \sup_{t,s, X,Y,Y \neq 0 } \frac{|f(X+Y,t) -2
f(X,t)-f(X-Y,t)|^p}{|Y|^{p\al}}.\nonumber
\end{align}
if $ p = \infty$.

\item[(4)]
The operator $S_\te: {\mathcal L}^p_{\al} \ri {\mathcal L}^p_{\al +\te}, \,\, S_\al f = H_\al * f$ is isomorphism for all $\al, \, \te \in {\bf R}$
  and $1 \leq p \leq \infty$.

\item[(5)]
${\mathcal S}({\bf R}^{n+1})$ is dense subset of ${\mathcal
L}^p_\al({\bf R}^{n+1})$ for all $\al \in {\bf R}$ and $1 \leq p
\leq \infty$.

\item[(6)]
${\mathcal L}^p_{\al_1} ({\bf R}^{n+1}) \subset {\mathcal
L}^p_{\al_2} ({\bf R}^{n+1}) $ for $\al_2 < \al_1$.
 \end{itemize}
\end{prop}
For the details of the proof of Proposition \ref{prop2} we refer
\cite{BL} for $(2)$ (in particular Definition 6.2.2, Theorem 6.2.4
and Theorem 6.4.5 in \cite{BL}), and  refer \cite{DT2} (Theorem 3)
for $(3)$. It is not difficult to derive (4) -(6) (see \cite{BL}).


For the sake of later use, we define  $L^p(\R)$- multiplier (
$L^p({\bf R}^{n+1})$- multiplier) as follows.
\begin{defin}
We say  that $\mu \in {\mathcal S}' (\R)$  is
$L^p(\R)$-multiplier if
\begin{align}\label{multiplier norm}
\| {\mathcal F}^{-1} (\mu \hat f) \|_{L^p (\R)} \leq M \| f\|_{L^p
(\R)}
\end{align}
for all $f \in {\mathcal S}(\R)$, where ${\mathcal F}^{-1} (f)$ is
the inverse Fourier transform of $f$. We call the minimal constant $M$ satisfying \eqref{multiplier norm} $L^p$-mutiplier norm of $\mu$.
\end{defin}

Similarly, we define $L^p({\bf R}^{n+1})$-multiplier. We introduce the Marcinkiewicz multiplier theorem  (see Theorem $4.6^{'}$ in  \cite{S}).
\begin{prop}\label{Marcinkiwitz}
Let $m$ be a bounded function on ${\bf R}^n \setminus \{ 0\}$.
Suppose also
\begin{itemize}
\item[(a)]
$|\mu( \xi)| \leq B$,\\
\item[(b)]
for each $0 < k \leq n$,
\begin{align*}
\sup_{\xi_{k+1}, \cdots \xi_{n}} \int_{\rho} |\frac{\pa^k \mu}{\pa \xi_1 \pa \xi_2  \cdots \pa \xi_k}| d\xi_1 \cdots d\xi_k \leq B
\end{align*}
as $\rho$ ranges over dyadic rectangles of ${\bf R}^k$  (If $k =n$, the " $\sup$ " sign is omitted).
\item[(c)]
The condition analogous to (b) is valid for every for one of the $n
!         $ permutations of the variables $\xi_1, \, \xi_2, \,
\cdots \xi_n$.
\end{itemize}
Then $mu$ is $L^p$-multiplier, $1 < p<\infty$ and the multiplier
norm depend only on  $B, \, p$ and $n$.
\end{prop}





We denote by $D^i_{X_k}, \, i \in {\bf N} \cup \{ 0\} $  the $i$ times  derivatives with respect to $X_k$.
When $i =1$, we denote $D^1_{X_k} = D_{X_k}$.  We also denote
$D^{\be}_{X}$  by  the $D_{X_1}^{\be_1} \cdots D_{X_n}^{\be_n}$   for $\be \in ({\bf N} \cup \{0\})^n$.
We denote by $D^\frac12_t $ the pseudo-differential operator whose Fourier
transform is defined by $\widehat{D^\frac12_t f}(\tau) =
|\tau|^{\frac12} \hat{f} (\tau)$ for complex-valued function $f$.
It is well-known that
\begin{align}\label{half}
D^{\frac12}_t f(t) = c \int_{{\bf R}} \frac{f(t) -
f(s)}{|t-s|^{\frac32}} ds
\end{align}
for complex-value function $f$.   For non-negative integer, we also denote $ D^{ i}_t f$  by $i$ times derivatives of
$f$  and $ D^{ i +  \frac12}_t f$  by $D^\frac12_t D^{i}_t f$, respectively.
Note that  $ D_tf = HD^\frac12 D^\frac12 f$.

\section{The properties of parabolic Sobolev and parabolic Besov spaces }
\setcounter{equation}{0} \label{sec1-2}

In this section, we study the properties  of parabolic Sobolev and
parabolic Besov spaces.

\begin{theo}\label{iterates}
Let $ 1 < p < \infty$ and $\al \in {\bf R}$. Then $f \in {\mathcal
L}^p_\al ({\bf R}^{n+1})$ if and only if $f, D_{X_k} f,
D_t^{\frac12} f \in {\mathcal L}^p_{\al -1} ({\bf R}^{n+1})$ for
all $1 \leq k \leq n$. Furthermore,
\begin{eqnarray}\label{Iterates}
\|f\|_{ {\mathcal L}^p_\al}\approx \| f\|_{{\mathcal L}^p_{\al
-1}} + \sum_{ 1 \leq k \leq n}\|D_{X_k} f\|_{{\mathcal
L}^p_{\al-1}} + \|D_t^{\frac12} f\|_{{\mathcal L}^p_{\al-1}}.
\end{eqnarray}
\end{theo}

\begin{proof}\
First, we assume $\al =1$. Suppose  $f \in {\mathcal
L}^p_{1} ({\bf R}^{n+1})$ so that $f = H_{1} * g$ for some $g
\in L^p({\bf R}^{n+1})$. Then,  for $1 \leq k \leq n$,  we have
\begin{align}\label{0620}
\begin{array}{ll}\vspace{2mm}
\widehat{D_{X_k} f }
= \frac{-2\pi \xi_k}{(1 +4 \pi^2|\xi|^2 + i \tau)^{\frac12}}
                                   \hat g, \quad
\widehat{D^t_{\frac12}  f}
=  \frac{|\tau|^{\frac12}}{(1 + 4 \pi^2|\xi|^2 + i \tau)^{\frac12}}
            \hat{g}.
\end{array}
\end{align}
Applying Proposition \ref{Marcinkiwitz}, we have that $ \mu_{1k} (\xi, \tau)
= \frac{-2\pi \xi_k}{(1 + 4 \pi^2|\xi|^2 + i \tau)^{\frac12}} ,
\,\, \mu_2(\xi, \tau)=\frac{|\tau|^{\frac12}}{(1 + 4 \pi^2 |\xi|^2
+ i \tau)^{\frac12}} $ are $L^p({\bf R}^{n+1})$ multipliers for $1
< p < \infty$. Then,  from \eqref{0620}, we  get
\begin{align*}
\|D_{X_k} f\|_{L^p} &= \|{\mathcal F}^{-1}(  \mu_{1k}(\xi, \tau) \hat g )\|_{L^p} \lesssim
 \|g\|_{L^p} =  \|f\|_{{\mathcal L}^p_1}  \quad 1 \leq k \leq n,\\
\|D_t^{\frac12} f\|_{  L^p}& = \| {\mathcal F}^{-1}(  \mu_2(\xi,
\tau) \hat g ) \|_{L^p} \lesssim \|g\|_{L^p} = \|f\|_{{\mathcal
L}^p_1}.
\end{align*}
From (6) in Proposition \ref{prop2}, we obtain $\| f\|_{L^p}
\lesssim\| f\|_{{\mathcal L}^p_1}$. Hence, we proved the one-side of
Theorem \ref{iterates}.

Now, we prove the converse inequality. Suppose
$ f, \,\, D^\frac12_t f , \,\, D_{X_k} f\in L^p({\bf R}^{n+1}), \,\, 1 \leq k \leq n$.
We claim that $f = H_1 * g$ for some $g \in L^p({\bf R}^{n+1})$ satisfying
\begin{align}\label{0627}
\| g\|_{L^p} \lesssim   \big( \| f\|_{L^p} + \sum_{1\leq k\leq n}
\|D_{X_k} f\|_{L^p} + \| D^\frac12_t f\|_{L^p} \big).
\end{align}
If then, $f = H_1 * g\in {\mathcal L}^p_1 ({\bf R}^{n+1}) $ with $\|
f\|_{{\mathcal L}^p_1}\lesssim\big( \| f\|_{L^p} + \sum_{1\leq k\leq
n} \|D_{X_k} f\|_{L^p} + \| D^\frac12_t f\|_{L^p} \big)$, and  this
will complete the proof of Theorem \ref{iterates}.

To prove the claim, let us  $R_k, \,\, 1 \leq k \leq n$ be Riesz transforms in $\R$. Then, we have
\begin{align*}
{\mathcal F}^{-1}(( 1 + |\xi| + |\tau|^{\frac12 }) \hat f) = f +
\sum_{1 \leq k \leq n} R_k \frac{\pa f}{\pa x_k} + D^{\frac12}_t f
\in L^p ({\bf R}^{n+1}).
\end{align*}
 Set $\hat K (\xi,\tau) = \frac{(1 + 4 \pi^2|\xi|^2 +
i\tau)^{\frac12}}{1 + |\xi| + |\tau|^{\frac12}}$  and $g = K
*\Big(f + \sum_{1 \leq k \leq n} R_k \frac{\pa f}{\pa x_k} +
D^{\frac12}_t f \Big) $.  Applying
Proposition \ref{Marcinkiwitz}, we have that  $ \hat K (\xi,\tau)$ is
$L^p({\bf R}^{n+1})$-multiplier. Hence we have $g \in L^p({\bf
R}^{n+1})$. Hence, \eqref{Iterates} holds for $\al =1$.

For general $\al \in {\bf R}$, by  (4) in  Proposition \ref{prop2},
we have that $S_{\al -1}:{\mathcal L}^p_1 \ri {\mathcal L}^p_\al$
and $S_{\al-1}: L^p \ri {\mathcal L}^p_{\al-1}$ are isomorphism
whose inverses are $ S^{-1}_{\al -1} = S_{-\al +1}$. Note that
$D_{X_k} S_{-\al +1} f =  S_{-\al +1} D_{X_k} f $ and  $
D^\frac12_tS_{-\al +1} f = S_{-\al +1}  D^\frac12_t f$. Hence, we
get
\begin{align*}
f \in {\mathcal L}^p_\al &   \Leftrightarrow  S^{-1}_{\al -1 } f = S_{-\al +1} f \in {\mathcal L}^p_1\\
 &  \Leftrightarrow  S_{-\al +1} f, \,\,   D_{X_k} S_{-\al +1} f (= S_{-\al +1} D_{X_k}  f),  \,\,   D^\frac12_tS_{-\al +1} f (= S_{-\al +1}   D^\frac12_t f) \in L^p\\
 &  \Leftrightarrow f, \,\, D_{X_k} f, \,\, D^\frac12_t f \in {\mathcal L}^p_{\al-1}.
\end{align*}
Hence, we complete the proof of \eqref{Iterates}.
\end{proof}

\begin{coro}\label{iterates2}
Let $ 1 < p < \infty$ and $\al \in {\bf R}$. Then $f \in {\mathcal
L}^p_2 ({\bf R}^{n+1})$ if and only if $f, \,
D_{X_k} D_{X_l} f, \,D_t f \in   L^p
({\bf R}^{n+1})$ for all $1 \leq k,l \leq n$. Furthermore,
\begin{align}
\begin{array}{ll}\vspace{2mm}
\|f\|_{ {\mathcal L}^p_\al} \approx  \| f \|_{{\mathcal L}^p_{\al
-2}}
+ \sum_{0 \leq k,l\leq n}\|D_{X_k} D_{X_l} f\|_{{\mathcal
L}^p_{\al-2}} + \|D_ f\|_{{\mathcal L}^p_{\al-2}}.
\end{array}
\end{align}
\end{coro}
\begin{proof}
As the proof of Theorem \ref{iterates}, it suffices to show the Corollary when $\al =2$.
Suppose  $f \in {\mathcal L}^2_p({\bf R}^{n+1})$. Since ${\mathcal L}_0^p({\bf R}^{n+1})= L^p({\bf R}^{n+1})$,
applying the Theorem \ref{iterates}  two times, we have
 \begin{align}\label{equiv3-1}
 \begin{array}{ll}\vspace{2mm}
  & \|   f\|_{ {\mathcal L}^p_2} \approx
 \| f \|_{ L^p }
+ \sum_{0 \leq k\leq n}\|D_{X_k } f\|_{L^p }
+ \|D_t^\frac12 f\|_{L^p }
+ \sum_{0 \leq k,l\leq n}\| D_{X_kX_l } f\|_{ L^p } \\
& \hspace{30mm} + \sum_{0 \leq k\leq n}\|D_t^\frac12 D_{X_k} f\|_{ L^p }
+ \|D_t  f\|_{ L^p }.
\end{array}
 \end{align}
 Hence, if $f \in  {\mathcal L}^p_2({\bf R}^{n+1})$, then we have
\begin{align*}
 \| f \|_{L^p }
+ \sum_{0 \leq k,l\leq n}\|D_{X_k} D_{X_l} f\|_{L^p } + \|\frac{\pa
f}{\pa t}\|_{L^p } \lesssim \|f\|_{ {\mathcal L}^p_2}.
\end{align*}
Conversely, suppose that $ f,  \,\, D_{X_k} D_{X_l} f, \, D_t f \in
L^p ({\bf R}^{n+1})$. Note that applying Proposition
\ref{Marcinkiwitz},  we have that  $\nu_1(\xi,
\tau)=\frac{|\tau|^\frac12}{1 + 4\pi^2|\xi|^2 + i\tau}, \,\,
\nu_{2,k}(\xi, \tau)=\frac{2\pi i \xi_k}{1 + 4\pi^2|\xi|^2 + i\tau},
 \,\, \nu_{3,k}(\xi, \tau) =\frac{2\pi i \xi_k |\tau|^\frac12}{1 + 4\pi^2|\xi|^2 + i\tau},$  $ 1 < p < \infty     $ are $L^p({\bf R}^{n+1})$-multipliers.
 Then, we have
 \begin{align} \label{0628}
 \begin{array}{ll}\vspace{2mm}
 &\widehat{D_t^\frac12 f} =  \nu_1(\xi, \tau) (1 + 4\pi^2|\xi|^2 + i\tau) \hat f, \,\,
 \widehat{D_{X_k} f} =  \nu_{2,k}(\xi, \tau) (1 + 4\pi^2|\xi|^2 + i\tau) \hat f,\\
 & \hspace{30mm}  \widehat{D^\frac12_t D_{X_k} f} = \nu_{3,k}(\xi, \tau) (1 + 4\pi^2|\xi|^2 + i\tau) \hat f.
 \end{array}
 \end{align}
 Note that $ {\mathcal F}^{-1}((1 + 4\pi^2|\xi|^2 + i\tau) \hat f )  = f + \sum_{1 \leq k \leq n} D^2_{X_k} f + D_t f $.
 Hence, from \eqref{0628}, we have
\begin{align}\label{0620-2}
\| D^\frac12_{t} f\|_{ L^p}  + \| D_{X_k} f\|_{ L^p}  + \|
D^\frac12_{t}  D_{X_k}f\|_{  L^p}     \lesssim \big( \|  f\|_{ L^p}
+ \| D_{t} f\|_{  L^p}  + \sum_{1 \leq k, l \leq n}\| D_{X_k} D_{
X_l }f\|_{  L^p}\big).
\end{align}
With   \eqref{equiv3-1}, \eqref{0620-2} and the assumption,  this implies
\begin{align*}
\|f\|_{ {\mathcal L}^p_2} \lesssim\big( \| f \|_{L^p} + \sum_{0 \leq
k,l\leq n}\|D_{X_k} D_{X_l} f\|_{ L^p} + \|D_t f\|_{  L^p}\big).
\end{align*}
Hence, we completed the proof of Corollary \ref{iterates2}.
\end{proof}

Now, we define parabolic Sobolev space $ \tilde W^{\al, \frac12
\al}_p({\bf R}^{n+1})$ and  $W^{2\al,   \al}_p({\bf R}^{n+1})$ for
positive integer $\al$ and $1 \leq p \leq \infty$ by
\begin{align*}
\tilde W^{\al, \frac12 \al}_p({\bf R}^{n+1})  :& =   \{ f \in L^p({\bf
R}^{n+1}) \, | \,  \,\, D_X^\be  D^{\frac{l}2}_t f  \in L^p({\bf
R}^{n+1}), \quad |\be| +  l   \leq \al \,\, \},\\
 W^{2\al,   \al}_p({\bf R}^{n+1})  : &=   \{ f \in L^p({\bf
R}^{n+1}) \, | \,  \,\, D_X^\be  D^l_t f  \in L^p({\bf
R}^{n+1}), \quad |\be| + 2 l   \leq 2\al \,\, \}
\end{align*}
with norms
\begin{align*}
 \|f\|_{\tilde W^{\al, \frac12 \al}_p }: =   \sum_{|\be| + \frac12l \leq \al} \| D^\be_X D^{\frac12 l}_t f\|_{L^p}, \quad
 \|f\|_{W^{ 2\al, \al}_p }: =   \sum_{|\be| +  l \leq 2\al} \| D^\be_X D^{  l}_t f\|_{L^p}.
\end{align*}

\begin{rem}
\begin{itemize}
\item[(1)]
From the Theorem \ref{iterates} and Corollary \ref{iterates2},
 if $\al$ is non-negative integer and $1 < p < \infty$, then we have
\begin{align}
{\mathcal L}^p_\al ({\bf R}^{n+1}) =  \tilde W^{\al, \frac12 \al}_p ({\bf R}^{n+1}),
\quad  {\mathcal L}^p_{2\al} ({\bf R}^{n+1}) =  \tilde W^{2\al,    \al}_p ({\bf R}^{n+1}) = W^{2\al,    \al}_p ({\bf R}^{n+1})
\end{align}
with the equivalent norms.

\item[(2)]When $p = 1$ or $p= \infty$, the spaces ${\mathcal L}^p_\al ({\bf R}^{n+1})$
and $\tilde W^{\al, \frac12 \al}_p ({\bf R}^{n+1}) $ are different spaces, and  ${\mathcal L}^p_{2\al} ({\bf R}^{n+1})$,
$ \tilde W^{2\al,   \al}_p ({\bf R}^{n+1}) $
and $  W^{2\al,    \al}_p ({\bf R}^{n+1})$ are different spaces each other.
\end{itemize}
\end{rem}

Next, we study about the properties of parabolic Besov spaces.
\begin{theo}\label{iterateb}
Let $ 1 \leq p \leq \infty$ and $\al \in {\bf R}$. Then $f \in
{\mathcal B}^{\al, \frac12 \al}_p ({\bf R}^{n+1})$ if and only if $f, D_{X_k} f,
D_t^{\frac12} f \in {\mathcal B}_p^{\al -1, \frac12 \al -\frac12} ({\bf R}^{n+1})$ for
all $1 \leq k \leq n$. Furthermore,
\begin{eqnarray}\label{iteratedb}
\|f\|_{ {\mathcal B}^{\al, \frac12 \al}_p} \approx  \|f\|_{{\mathcal B}^{\al-1, \frac12 \al -\frac12}_p}
+ \sum_{1 \leq k \leq n}\|D_{X_k} f\|_{{\mathcal B}^{\al-1, \frac12 \al -\frac12}_p} +
\|D^{\frac12}_tf\|_{{\mathcal B}^{\al-1, \frac12 \al -\frac12}_p}.
\end{eqnarray}
\end{theo}

\begin{proof} If $1 < p <\infty$, then by Theorem \ref{iterates} and the property of
interpolation spaces (see (2) of Proposition \ref{prop2}),  \eqref{iteratedb} holds.
Hence we have only to consider the critical case $p=1$ and
$p=\infty$. Since the proofs are exactly same, we only prove in the case of $p =1$.

Suppose that  $ f \in {\mathcal B}^{\al, \frac12 \al }_1({\bf
R}^{n+1})$. Then by the definition of the parabolic Besov space,
we have
\begin{eqnarray*}
\|f\|_{{\mathcal B}^{\al, \frac12 \al}_1} = \|f * \psi\|_{L^1} + \sum_{1 \leq i <
\infty} 2^{\al i} \|f * \phi_i\|_{L^1} < \infty.
\end{eqnarray*}
Note that by construction of  $ \psi$ and $\phi_i$ in section \ref{sec1},  we have
$\hat \psi + \hat \phi_1 + \hat \phi_2=1$ in $supp \, (\hat \psi + \hat \phi_1)$ and
$\hat \phi_{i-1} + \hat \phi_i + \hat \phi_{i+1}=1$ in $supp \,  \hat \phi_i$ for $i \geq 2$.
Hence, using $D_{X_k} (f*g) = (D_{X_k} f)* g = f * ( D_{X_k} g)$,  we have
\begin{align*}
(D_{X_k} f )* \psi  &= f * \psi  * D_{X_k}( \psi +
\phi_1 + \phi_2),\\
 (D_{X_k} f) * \phi_1 & = f * \phi_1 * D_{X_k}( \psi +
 \phi_1 + \phi_2),\\
 (D_{X_k} f )* \phi_i & = f * \phi_i
* D_{X_k}( \phi_{i-1} + \phi_i + \phi_{i+1}),\quad i \geq 2.
\end{align*}
Note that  $\|D_{X_k} \psi\|_{L^1} \lesssim$ and $ \|D_{X_k} \phi_i
\|_{L^1} \lesssim2^i$.     Hence, by Young's inequality, we have
\begin{align*}
\| (D_{X_k} f )* \psi   \|_{L^1}  & \leq  \|  f * \psi \|_{L^1} \|  D_{X_k}( \psi +
\phi_1 + \phi_2) \|_{L^1} \lesssim\|  f * \psi \|_{L^1},\\
\| (D_{X_k} f )* \phi_1   \|_{L^1}  & \leq  \|  f * \phi_1 \|_{L^1} \|  D_{X_k}( \psi +
\phi_1 + \phi_2) \|_{L^1} \lesssim\|  f * \phi_1\|_{L^1},\\
\| (D_{X_k} f )* \phi_i   \|_{L^1} &  \leq  \|  f * \phi_i\|_{L^1}
\|  D_{X_k}( \phi_{i-1} + \phi_{i} + \phi_{i-1}) \|_{L^1} \lesssim
2^i\|  f * \phi_i \|_{L^1}, \quad i \geq 2.
\end{align*}
Hence, we have
\begin{align*}
\|D_{X_k} f\|_{{\mathcal B}^{\al -1,\frac12 \al -\frac12}_1} &= \|D_{X_k} f *
\psi\|_{L^1} + \sum_{1 \leq i < \infty} 2^{(\al-1) i} \|D_{X_k} f * \phi_i\|_{L^1}\\
 &  \lesssim \big( \| f * \psi\|_{L^1} +  \sum_{1 \leq i < \infty} 2^{\al i}  \| f *
                \phi_i\|_{L^1} \big)\\
 & =  \|f\|_{{\mathcal B}^{\al-1, \frac12 \al -\frac12}_1}.
\end{align*}
Similarly, we obtain
\begin{align*}
(D^{\frac12}_t f) * \psi  &= f * \psi  * D^\frac12_t( \psi +
\phi_1 + \phi_2),\\
 (D^{\frac12}_t f) * \phi_1 & = f * \phi_1 * D^\frac12_t( \psi +
 \phi_1 + \phi_2),\\
 (D^{\frac12}_t f) * \phi_i & = f * \phi_i
* D^\frac12_t( \phi_{i-1} + \phi_i + \phi_{i+1}), \quad i \geq 2.
\end{align*}
Note that using \eqref{half} and change of variables, we have
\begin{eqnarray}\label{0618}
\begin{array}{ll}
\|D^\frac12_t \phi_i\|_{L^1} &= c\int_{{\bf R}^{n+1}}| \int_{{\bf
R}} \frac{\phi_i(X,t) - \phi_i (X,s)}{|t-s|^{\frac32}}ds|dXdt\\
  & \lesssim2^i \int_{{\bf R}^{n+1}} \int_{{\bf R}}
     \frac{|\phi (X,t) - \phi (X,s)|}{|t-s|^{\frac32}}dsdXdt \\
& \lesssim2^i\|\phi\|_{{\mathcal B}^{1,\frac12}_1 ({\bf R}^{n+1})}\\
\|D^{\frac12}_t \psi \|_{L^1} &\lesssim\| \psi\|_{{\mathcal
B}^{1,\frac12}_1 ({\bf R}^{n+1})}.
\end{array}
\end{eqnarray}
As the same reason to the case of $D_{X_k} f$, using Young's
inequality, we have $\|D^\frac12_t f\|_{{\mathcal B}^{\al-1,\frac12
\al -\frac12}_1} \lesssim\| f\|_{{\mathcal B}^{\al-1, \frac12 \al
-\frac12}_1}$. Hence, we proved one side of \eqref{iteratedb}.

Conversely, we suppose that $\|f\|_{{\mathcal B}^{\al -1,\frac12 \al -\frac12}_1} ,
\|D_{X_k}f\|_{{\mathcal B}^{\al -1,\frac12 \al -\frac12}_1},\|D^\frac12_t
f\|_{{\mathcal B}^{\al -1,\frac12 \al -\frac12}_1} < \infty.$  Since $\phi$ is
supported in $\{(\xi,\tau) \in {\bf R}^{n+1} \, | \, 2^{-1} <
|\xi| + |\tau|^\frac12 < 2 \}$, we have that $\frac{1}{(- 4\pi^2|\xi|^2
+i\tau)} \phi (\xi, \tau) \in {\mathcal S}({\bf R}^{n+1})$. We
define $\Phi$ and $\Phi_i$ by the functions whose Fourier
transforms are written by $ \hat \Phi(\xi, \tau)=\frac{1}{ -4\pi^2|\xi|^2
+i\tau} \phi (\xi, \tau) $ and $\hat \Phi_i(\xi, \tau) = \hat \Phi
(2^{-i}\xi, 2^{-2i} \tau)$. Then, for $i \geq 2$, we have
\begin{align}\label{0628-2}
\begin{array}{ll} \vspace{2mm}
\widehat{ f* \phi_i}  & =\hat f  \hat \phi_i
                  ( \hat \phi_{i-1} +\hat \phi_i + \hat \phi_{i+1})\\ \vspace{2mm}
  &=  \hat f  \hat \phi_i
               \frac{-4\pi^2|\xi|^2 + i\tau}{ -4\pi^2|\xi|^2 + i \tau} \big(\hat \phi_{i-1} +\hat \phi_i +
               \hat \phi_{i+1} \big) \\ \vspace{2mm}
 &=  2^{-2i} \sum_{1 \leq k \leq n} \widehat {D_{X_k}f}  \hat
 \phi_i   \big( \widehat{D_{X_k}\Phi_{i-1}}
                + \widehat {D_{X_k} \Phi_i}  +\widehat{D_{X_k}\Phi_{i+1}} \big)\\
 &  \quad + 2^{-2i}\widehat{D^\frac12_t f}  \hat \phi_i
                  \big( \widehat{HD^\frac12_t \Phi_{i-1}}
                + \widehat{HD^\frac12_t \Phi_i}  + \widehat{HD^\frac12_t
                \Phi_{i+1}} \big),
                \end{array}
\end{align}
where $H$ is the Hilbert transform. We used the fact that $D_t \Phi
= H D^{\frac12}_t D^{\frac12}_t \Phi$. Note that
$\|D_{X_k}\Phi_i\|_{L^1} \lesssim2^i$. Moreover,
\begin{align*}
H D^\frac12_t \Phi_i(X,t) &= \lim_{\ep \ri 0} \int_{\ep< |t-s| <
\frac{1}{\ep}} \frac{sign(t-s)}{|t-s|^{\frac32}}\Phi_i(X,s)ds\\
&= \lim_{\ep \ri 0} \int_{\ep< |t-s| < \frac{1}{\ep}}
\frac{sign(t-s)}{|t-s|^{\frac32}}(\Phi_i(X,s) - \Phi_i (X,t))ds,
\end{align*}
where $sign(t) =1$ if $t >0$ and $sign(t) =-1$ if $t < 0$.
Hence, using change of variables (see \eqref{0618}), we get
\begin{eqnarray*}
\|HD^\frac12_t\Phi_i\|_{L^1}  \lesssim\int_{{\bf R}^{n+1}}
\int_{{\bf R}} \frac{ |\Phi_i(X,s) - \Phi_i
(X,t)|}{|t-s|^{\frac32}}dsdXdt  \lesssim2^i \|\Phi\|_{{\mathcal
B}^{1,\frac12}_1({\bf R}^{n+1})}.
\end{eqnarray*}
Hence, applying  Young's inequality in \eqref{0628-2}, we have
\begin{align}\label{Sovk}
\| f* \phi_i\|_{L^1} \leq 2^{-i} (\sum_{1 \leq k \leq n}
\|D_{X_k}f
* \phi_i\|_{L^1} + \|D^{\frac12}_t f * \phi_i \|_{L^1}) \quad i \geq 2.
\end{align}





Hence by (\ref{Sovk}),  we have
\begin{align*}
\|f\|_{{\mathcal B}^{\al, \frac12 \al}_1} &= \|f * \psi\|_{L^1} + \sum_{1 \leq i <
\infty} 2^{\al i} \|f * \phi_i\|_{L^1} \\
& \lesssim\Big(  \|f * \psi\|_{L^1} + \| f* \phi_1\|_{L^1} +
\sum_{2 \leq i < \infty} 2^{ (\al -1) i}        \big(
\sum_{1 \leq k \leq n} \|D_{X_k}f
* \phi_i\|_{L^1} + \|D^{\frac12}_t f * \phi_i \|_{L^1}   \big)       \Big)\\
& \lesssim\Big(\|f\|_{{\mathcal B}^{\al -1,\frac12 \al -\frac12}_1}
+ \|D_{X_k}f\|_{{\mathcal B}^{\al -1,\frac12 \al -\frac12}_1} +
\|D^\frac12_tf\|_{{\mathcal B}^{\al -1,\frac12 \al
-\frac12}_1}\Big).
\end{align*}
Hence, we completed the proof of Theorem \ref{iterateb}.
\end{proof}

By (\ref{besovnorm1}),  (\ref{besovnorm2}) and  Theorem
\ref{iterateb}, we get the  following Corollary;
\begin{coro}\label{iterateb2}
\begin{itemize}

\item[(1)]
Let $ 1 \leq  p \leq   \infty$  and $\al \in {\bf R}$ such that $2i
< \al < 2i+2$ for positive integer $l$. Then $f \in {\mathcal
B}^{\al, \frac12 \al}_p ({\bf R}^{n+1})$ if and only if $f, \,
D^\be_{X}  f, D^i_t f  \in {\mathcal B}_p^{\al -2i,\frac12 \al -i}
({\bf R}^{n+1})$ for all $|\be| = 2i$. Furthermore,
\begin{align*}
\|f\|_{ {\mathcal B}^{\al , \frac12 \al}_p} \approx & \|f\|_{{\mathcal B}^{\al
-2i,\frac12 \al -i}_p}  + \sum_{ |\be| =2 i }\|D^\be_{X}  f\|_{{\mathcal
B}^{\al -2i,\frac12 \al -i}_p} + \|D^l_t f \|_{{\mathcal B}^{\al
-2i,\frac12 \al -i}_p}.
\end{align*}
\item[(2)]
In particular,  for $1 \leq p < \infty$,   we have
\begin{align*}
&\|f\|^p_{ {\mathcal B}^{\al, \frac12 \al}_p}
\approx \|f\|^p_{W^{2i, i }_p} +  \int_{\R}\!\! \int_{{\bf R}
\times {\bf R}}
 \frac{|D^i_t  f(X,t) -  D^i_t f(X,s)|^p}{|t-s|^{ 1 + p\frac12(\al -2i )}}dtdsdX \\
& \qquad +  \sum_{ |\be| =2i}  \int_{{\bf R}} \int_{\R \times \R }
\frac{|D^\be_{X} f(X+Y,t) -2 D^\be_{X}  f(X,t) +
D^\be_{X}  f(X-Y,t)|^p}{|Y|^{ n + p(\al -2i )}}dXdYdt
\end{align*}
and
\begin{align*}
 &\|f\|_{{\mathcal B}^{\al, \frac12 \al}_\infty} \approx  \|f\|_{ W^{2 i, i }_\infty}
 +\sup_{X ,t,s, t \neq s}
\frac{|  D^i_tf(X,t) - D^i_tf(X,s)|}{|t-s|^{
\frac12 (\al -2i )}} \\
& \hspace{10mm} +  \sum_{|\be|   =2 i }  \sup_{t,s, X,Y,Y \neq 0 }
\frac{|D^\be_{ X}  f(X+Y,t) -2D^\be_{X}
f(X,t)-D^\be_{X} f(X-Y,t)|}{|Y|^{\al -2 i }}.
\end{align*}
\end{itemize}
\end{coro}
\begin{proof}
 Applying Theorem \ref{iterateb}  two times, we obtain   one side of (1).
To show that the right side of (2) implies the left side of (3),  we replace $\frac{1}{ - 4\pi^2|\xi|^2
+i\tau} \phi (\xi, \tau) $ by    $\frac{1}{( - 4\pi^2|\xi|^2)^i
+ (i\tau)^i} \phi (\xi, \tau)  $ in   \eqref{0628-2} and apply the proof of Theorem \ref{iterateb}.
 (2) holds because of  (1)  and \eqref{prop2}.
\end{proof}

\section{Parabolic Sobolev and parabolic Besov space in  ${\bf R}^{n}_T$}
\setcounter{equation}{0} \label{sec4}

If $i$ is non-negative integer, we define the parabolic  Sobolev
space $W_p^{2i,i} ({\bf R}^n_T), \,\, 0 < T \leq \infty$ by
\begin{eqnarray*}
W_p^{2i,i} ({\bf R}^n_T) = \{ f \, | \, D^{\be}_{X}D_t^l f \in L^p ({\bf R}^n_T),
\,\, 0 \leq |\be | + 2l \leq 2i \},
\end{eqnarray*}
so that the norm in $W^{2i,i}_p ({\bf R}^n_T)$  is defined by
\begin{align*}
\|f\|_{W^{2i,i}_p({\bf R}^n_T)} &= \Big(\sum_{2l + |\be| \leq  2i}
\int \int_{{\bf R}^n_T} |D^{\be}_{X} D_t^l f(X,t) |^p dXdt\Big)^{\frac1p},
\quad 1 \leq
p < \infty,\\
\|f\|_{W^{2i,i}_\infty({\bf R}^n_T)} & =  \sum_{2l + |\be| \leq  2i}
\sup_{(X,t) \in {\bf R}^n_T } |D^{\be}_{X} D_t^l f(X,t)|, \quad p= \infty
\end{align*}
 Let $ 2i< \al <2i+2$. We would like to define parabolic  Besov space
${\mathcal B}^{\al,\frac12 \al}_p ({\bf R}^n_T).$   We say
that $f \in {\mathcal B}^{\al, \frac12 \al}_p ({\bf R}^n_T)$ if and only if
\begin{align*}
 & \|f\|^p_{W^{2i,i}_p ({\bf R}^n_T)}+\sum_{|\be| + 2l = 2i}\Big[ \int_{\R}\!\!
\int_0^T\int_0^T  \frac{|D_{X}^\beta D^l_t f(X,t)
- D_{X}^\beta D^l_t f(X,s)|^p}{|t-s|^{ 1 + \frac12p (\al
-2i)}}dtdsdX
 \\
 & \int_0^T  \int_{\R \times \R } \frac{|D_{X}^\beta
D^l_tf(X+Y,t) -2 D_{X}^{\beta}D^l_tf(X,t) + D_{X}^\beta D^l_t
f(X-Y,t)|^p}{|Y|^{ n + p(\al -2i)}}dXdYdt\Big] < \infty
\end{align*}
if $1 \leq p < \infty$ and
\begin{align*}
 & \|f\|_{W^{2i,i}_\infty ({\bf R}^n_T)}+\sum_{|\be| + 2l = 2i}\Big[ \sup_{X,t,s, t \neq s}
\frac{|D_{X}^\beta D^l_t f(X,t) - D_{X}^\beta D^l_t
f(X,s)|}{|t-s|^{  \frac12 (\al -2i)}}  \\
&+ \sup_{t,X,Y, Y\neq 0} \frac{|D_{X}^\beta D^l_tf(X+Y,t) -2
D_{X}^{\beta}D^l_tf(X,t) + D_{X}^\beta D^l_t f(X-Y,t)| }{|Y|^{\al
-2i}} \Big] < \infty.
\end{align*}

\begin{prop}\label{prop3}
Let $1 \leq p \leq \infty$. Suppose that there is a bounded linear
operator $ E_{ {\bf R}^n_T }: W^{2i,i}_p ({\bf R}^n_T) \ri {\mathcal
L}_{2i}^p ({\bf R}^{n+1})$ for all
non-negative integer $i$ and $1 \leq p \leq \infty$ such that 
 $E_{ {\bf R}^n_T} f =f$ in ${\bf R}^n_T$. Then
for $0 < \te < 1, \,\, i < l$, we get  $ (W^{2i,i}_p ( {\bf R}^n_T),
W^{2l,l}_p ({\bf R}^n_T))_{p,\te}={\mathcal B}^{\al,\frac12 \al}_p ({\bf R}^n_T)$, where $\al =
(1-\te) 2i + \te 2l$.
\end{prop}
\begin{proof}
Applying Theorem 4.12 and Corollary 4.13 in \cite{BS},  for $i \in
{\bf N}$, $1 \leq p \leq \infty$ and $0 < \te <1$, we obtain that
$(L^p({\bf R}^{n+1}), W^{2i, i}_p({\bf R}^{n+1}))_{\te p} =
{\mathcal B}^{2i \te, i \te}_p({\bf R}^{n+1}) $. Using the
Proposition 2.4 and the Proposition 2.17 in \cite{JK}, we obtain
Proposition \ref{prop3}.
\end{proof}

To apply the Proposition \ref{prop3},  we define extension operators from $W^{2i,i}_p ({\bf
R}^{n}_\infty)$ to ${\mathcal L}^p_{2i} ({\bf R}^{n+1})$ and from
$W^{2i,i}_p ({\bf R}^{n}_T)$ to ${\mathcal L}^p_{2i} ({\bf
R}^{n+1})$.
For $f \in W^{2i,i}_p ({\bf R}^{n}_\infty)$ we define extension $E_2
f$ of $f$ by
\begin{align}\label{extension}
E_2 f(X,t) = \left\{\begin{array}{l} f(X,t), \quad t\geq 0,\\
 \sum_{1 \leq j\leq 2i+1} \la_j f (X, -jt) \quad t \leq 0,
 \end{array}
 \right.
\end{align}
where the coefficients $\la_1, \cdots , \la_{2i+1}$ are the unique
solution of the $(2i+1) \times (2i+1)$ system of linear equations
$$
\sum_{1 \leq j \leq 2i+1} (-j)^l \la_j =1, \quad l = 0, 1, \cdots
,2 i.
$$
Then $E_2 f \in W^{2i,i}_p ({\bf R}^{n+1})$ with $ E_2f|_{{\bf
R}^{n}_\infty} = f, \,\, \|E_2 f\|_{W^{2i,i}_p ({\bf R}^{n+1})} \leq
c \| f\|_{W^{2i,i}_p ({\bf R}^{n}_\infty) }$  (see Theorem 4.26 in
\cite{A}).

We apply (\ref{extension}) to define the extension operator in
$W^{2i,i}_p (\R_T )$. Let $g \in W^{2i,i}_p (\R_T )$.
We define an extension $E_3$ by
\begin{align*}
E_3 g(X,t) =\te (t) \left \{\begin{array}{ll} \sum_{1 \leq j \leq
2i+1}
\la_j g (X,-jt) & \quad -T < t <0,\\
g(X,t)  & \quad 0 < t <T,\\
\sum_{1 \leq j \leq 2i+1} \la_j g (X,- j (2T-t)) & \quad T < t
< 2T.
\end{array}
 \right.
\end{align*}
and $E_3 g (X,t) =0$ otherwise, where $\te \in C^\infty_c ({\bf R})$
such that $\te \equiv 1 $ in $(0, T)$ and $supp \, \te \subset (-T,
2T)$. Then   $E_3 g|_{{\bf R}^{n}_T} = g$ and $\|E_3g \|_{W^{2i,i}_p
({\bf R}^{n+1})} \lesssim\| g\|_{W^{2i,i}_p (\R_T  )}$.

By Proposition \ref{prop3}, we have the  following theorem.
\begin{theo}\label{RBesov}
Then,
for $0 < \al $ and $1 \leq p \leq \infty$, ${\mathcal B}^{\al,
\frac12 \al}_p
(\R_T)$ is real interpolation space, that is, $(L^p(\R_T),
W_p^{2i,i} (\R_T))_{p,\te} = {\mathcal B}^{2(1-\te) i, (1 -\te) i}_p (\R_T)$, $0 < T \leq \infty$.
\end{theo}

\begin{theo}\label{iterates3}
Then,
for $\al \geq 2$, and $1 \leq  p \leq \infty$, $f \in {\mathcal
B}^{\al,\frac12 \al}_p (\R_T)$ if and only if $ f, \,D_{X_k} f, \,  D_{X_k} D_{ X_j }
f, \, D_t f \in {\mathcal B}^{\al -2,\frac12 \al -1}_p (\R_T)$, $0 < T \leq \infty$.
\end{theo}
\begin{proof} Because of the similarity of the proof, we consider
only the case of ${\bf R}^{n}_\infty$. We define extension operator,
\begin{align*}
E_4 f(X,t) =\left \{ \begin{array}{ll} f(X,t) \quad & t> 0,\\
\sum_{1 \leq j \leq 2i+1} (-j) \la_j  f(X,-jt)\quad  & t < 0.
\end{array}
\right.
\end{align*}
Then,  $E_4 : W^{2l-2,l-1}_p ({\bf R}^{n}_\infty) \ri W^{2l-2,l-1}_p
({\bf R}^{n+1}), \, 0 \leq l \leq i$ is bounded operator and so by
\eqref{RBesov}, we get $E_4 : {\mathcal B}^{\al, \frac12 \al}_p (\R_\infty) \ri
{\mathcal B}^{\al,\frac12 \al}_p ({\bf R}^{n+1}), \,\, \al > 0,\,\, 1 \leq p
\leq \infty$ is bounded operator. Note that
\begin{align}\label{iterates4}
D_{X_k} (E_2 f) = E_2  (D_{X_k} f), \,\,  D_{X_i }D_{X_k} ( E_2 f) =
E_2(D_{X_i}D_{X_k} f) ,\,\, D_t( E_2 f) = E_4 (D_t f).
\end{align}
Let $f \in {\mathcal B}^{\al, \frac12 \al}_p ({\bf R}^{n+1}_+)$. Then $E_2 f \in
{\mathcal B}^{\al, \frac12 \al}_p ({\bf R}^{n+1})$ and by Corollary
\ref{iterates2}, we have
$$
E_2 f, \,\, D_{X_k}( E_2 f), \,\, D_{X_i X_k}( E_2 f), \,\, D_t(
E_2 f) \in {\mathcal B}^{\al -2,\frac12 \al -1}_p ({\bf R}^{n+1}).
$$
Hence by (\ref{iterates4}), we have
\begin{align}\label{down}
 f, \,\, D_{X_k} f, \, D_{X_i}D_{ X_k} f, \,\, D_t f \in {\mathcal B}^{\al -2,\frac12 \al -1}_p ({\bf
R}^{n}_\infty).
\end{align}

Conversely, suppose that  \eqref{down} is true. Then
$$
E_2f, \, E_2 D_{X_k} f, \, E_2 D_{X_i}D_{X_k}f, \, E_4 D_t f \in
{\mathcal B}^{\al -2,\frac12 \al -1}_p ({\bf R}^{n+1}).
$$
By (\ref{iterates4}) and Corollary \ref{iterateb2}, we have $E_2 f
\in {\mathcal B}^{\al, \frac12 \al}_p ({\bf R}^{n+1})$. Hence $E_2 f|_{\R_\infty} = f
\in {\mathcal B}^{\al, \frac12 \al}_p(\Om).$
\end{proof}

\begin{rem}\label{remark}
Let $\al \geq 1$. If $ u \in {\mathcal B}^{\al, \frac12 \al}_p ({\bf
R}^{n}_T)$, $0 < T\leq \infty$. Combining Theorem \ref{RBesov} and Theorem
\ref{iterates3} we obtain the estimate
\begin{align}
\| D_X u \|_{{\mathcal B}^{\al-1, \frac12 \al -\frac12}_p ({\bf
R}^{n}_T)} \lesssim\| u \|_{{\mathcal B}^{\al,\frac12 \al}_p ({\bf
R}^{n}_T)}.
\end{align}
\end{rem}

\section{Proofs of Theorem \ref{mainresult}  }
\setcounter{equation}{0} \label{sec5}

 In this section, we study the relation of usual Besov spaces ${\mathcal B}_p^{\al-\frac2p}(\R)$
and parabolic Besov spaces ${\mathcal B}^{\al,\frac12 \al}_p ({\bf R}^{n}_T)$.

\begin{theo}\label{frac3p}
Let $0 < T < \infty$. Let $ f \in {\mathcal B}_p^{- \frac2p}(\R)$ and
$u$ be defined by (\ref{main4}). Then, for $ 1 \leq p \leq  \infty$, we have
\begin{align}\label{boundary2}
\| u\|_{L^p(\R_T)} \lesssim \|
f\|_{{\mathcal B}^{-\frac2p}_p (\R)}.
\end{align}
\end{theo}
(Compare with the section 1.8.1  in  \cite{Tr3}).

We introduce a function $\phi' \in {\mathcal S} ({\bf R}^{n})$, the Schwartz space in $\R$,
such that
\begin{eqnarray*}
\left\{\begin{array}{ll}
\hat{ \phi}'(\xi) > 0,  & \mbox{ on } 2^{-1} <
|\xi| < 2,\\
\hat{\phi}' (\xi) = 0, &\mbox{ elsewhere},
\end{array}
\right. \\
 \sum_{-\infty < i <
\infty} \hat{\phi}'(2^{-i}\xi) =1 ,&   ( \xi \neq 0).
\end{eqnarray*}
We define functions  $\phi'_i, \,\,\psi' \in {\mathcal S}(\R)$ whose Fourier transforms are
written by
\begin{eqnarray}\label{psi}
\begin{array}{ll}
\hat{\phi'_i}(\xi) &= \hat{\phi}'(2^{-i} \xi), \quad i = 0, \pm 1, \pm 2 , \cdots,\\
\hat{\psi'}(\xi) & = 1- \sum_{1 \leq i < \infty} \hat{\phi}' (2^{-i}
\xi).
\end{array}
\end{eqnarray}
As we defined the parabolic Besov space, we define a Besov space in $\R$.
For $\al \in {\bf R}$ we define the Besov space
${\mathcal B}^{\al,\frac12 \al}_{pq} ({\bf R}^{n})$ by
\begin{eqnarray*}
{\mathcal B}^{\al,\frac12 \al}_{pq} ({\bf R}^{n}) = \{ f \in {\mathcal
S}^{'}({\bf R}^{n}) \, | \, \|f\|_{{\mathcal B}^{\al, \frac12 \al}_{pq}} <
\infty \, \}
\end{eqnarray*}
with the norms
\begin{align*}
 \|f\|_{{\mathcal B}^{\al, \frac12 \al}_{pq}} :&  = \| \psi' * f\|_{L^p} + (
\sum_{ 1 \leq i  < \infty} (2^{\al i} \|\phi'_i *
f\|_{L^p})^q)^{\frac1q}, \quad 1 \leq q < \infty,\\
 \|f\|_{{\mathcal B}^{\al, \frac12 \al}_{p\infty}}  :& =  \sup (\| \psi' * f\|_{L^p} , \,\,  2^{\al i} \|\phi'_i *
f\|_{L^p}),
\end{align*} where $*$ is a convolution in ${\bf
R}^{n}$. When $p=q$, we simply denote ${\mathcal B}^{\al,\frac12 \al}_{pp} $
by ${\mathcal B}^{\al,\frac12 \al}_p$.

\begin{lemm}\label{multiplier}
Let $\hat{\Psi}'(\xi) = \hat{\psi}'(\xi) + \hat{\phi}'( 2^{-1}\xi)  + \hat{\phi}'( 2^{-2}\xi)$ and
$\hat{\Phi}' (\xi) = \hat{\phi}'(2^{-1} \xi) + \hat{\phi}' (\xi) + \hat{\phi}' (2\xi)$.
Let $\hat{\Phi}^{'}_i (\xi) = \hat{\Phi}^{'}(2^{-i} \xi), \,\, i \geq 2$ and let $\rho_{ti}(\xi) =
\hat{\Phi}^{'}_i ( \xi) e^{-t|\xi|^2}$ for each integer $i\geq 2$. Then,
$\rho_{ti}( \xi)$s' are $L^p({\bf R}^{n})$-multipliers with norms
$M(t,i)$ for $1 \leq p \leq \infty$. Furthermore, for $t
> 0$
\begin{align}\label{multiplier2}
M(t,i) & \lesssim e^{-\frac14 t2^{2i}}\sum_{0 \leq l \leq L} t^l
2^{2il} \lesssim e^{-\frac18 t2^{2i}},
\end{align}
where $L =[\frac{n}2] +1$.
\end{lemm}
\begin{proof}
Let $t > 0$.
The $L^p(\R)$-multiplier norms $M(t,i)$ of $\rho_{ti}(\xi) $ are
equal to $L^p(\R)$-multiplier norms of $\rho_{ti}^{'}(\xi)  =
\hat{\Phi}^{'}(\xi) e^{-t2^{2i} |\xi|^2}$ (see Theorem 6.1.3 in
\cite{BL}). To prove our lemma, we make use of the Lemma 6.1.5 in
\cite{BL}. Let $\be = ( \be_1, \cdots , \be_n),$ where $\be_i$ are
non-negative integers. Then, we have
\begin{align*}
|D^\be_{\xi} \rho_{ti}^{'}(\xi)| &\lesssim  e^{-\frac14t 2^{2i}} \sum_{0 \leq l \leq
|\be|} t^l 2^{2il} \chi_{\frac14 < |\xi| < 4}
(\xi),
\end{align*}
where $\chi$ is a characteristic function. Let $L=[\frac{n}2] + 1$ and $\te = \frac{n}{2L}$. Then by Lemma
6.1.5 in \cite{BL}, the $L^p(\R)$-multiplier norms of $\rho^{'}_{ti}$
are dominated by
\begin{align*}
\|\rho'_{ti}\|_{L^2(\R)}^{1 - \te} \sup_{|\be|= L} \|D^\be
\rho^{'}_{ti} \|_{L^2(\R)}^\te
&\lesssim e^{-\frac14t 2^{2i}}
\sum_{0 \leq l \leq L} t^l 2^{2il}.
\end{align*}

This completes the proof. \end{proof}

\begin{proof}[Proof of Theorem \ref{frac3p}]
Since the proof is similar, we only show in the case $1 \leq p <
\infty$. To prove Theorem \ref{frac3p}, we use $\hat{\psi}' (\xi) +
\sum_{1 \leq i < \infty} \hat{\phi}'(2^{-i} \xi) =1$ for all $\xi \in
\R.$ Note that
\begin{align*}
\hat u(\xi,t) = \big(  \hat{\Psi}'(\xi)  \hat{\psi}' (\xi) +  \hat{\Psi}'(\xi)  \hat{\phi}' (2^{-1} \xi)  \big)  e^{-t|\xi
|^2} \hat f +
\sum_{i=2}^\infty \hat{\Phi}'(2^{-i} \xi) \hat{\phi}' (2^{-i} \xi) e^{-t|\xi
|^2} \hat f.
\end{align*}
where $\hat u$ is the Fourier transform in $\R$.
Hence,
we have
\begin{align*}
 \int_0^T \int_{\R}  | u(X,t)|^pdXdt
&  \leq   c_p\int_0^T \int_{\R}  |{\mathcal F}^{-1}
\Big( \big(\hat{\Psi}'(\xi)  \hat{\psi}'(\xi) + \hat{\Psi}' (\xi) \hat{\phi}'_1
(\xi)     \big)            e^{-t|\xi|^2}  \hat{f}                  \Big)|^p dXdt\\
& \quad  +
c_p\int_0^T \int_{\R} |{\mathcal F}^{-1} \Big(\sum_{2 \leq i <
\infty}
\hat{\Phi}'_i(\xi) e^{-t|\xi|^2} \hat{\phi}'_i(\xi) \hat{f} \Big)|^p dXdt.
\end{align*}
 Note that by Young's inequality, we have
\begin{align}\label{norm}
\int_{\R} |\Ga(\cdot, t) *\Psi'| dX \leq  \int_{\R} |\Psi' (X)| dX  < \infty,
\end{align}
Applying Young's inequality again,  the first term is dominated by
\begin{align}\label{negative2}
\int_0^T \big(\| f * \psi' \|^p _{L^p (\R)} + \| f *
\phi'_1 \|^p _{L^p (\R)} \big) dt.
\end{align}
Since $\ \Phi^{'}_i(  \xi) e^{-t |\xi|^2}$ are
$L^p(\R)$-multipliers with norms $M(t,i)$ (see lemma
\ref{multiplier}), we have
\begin{align*}
&  \int_0^T \int_{\R}  |{\mathcal F}^{-1}
\Big(\sum_{ 1 \leq i <
\infty} \psi^{''}_i( \xi) e^{-t|\xi|^2} \phi'_i(\xi) \hat{f} \Big)|^p dXdt\\
& \leq  \int_0^T \Big(\sum_{t2^{2i} \leq 1} M(t,i) \| f * \phi'_i\|_{L^p}
\Big)^pdt\\
&  \quad + \int_0^T  \Big(\sum_{t2^{2i} \geq 1} M(t,i) \| f *
\phi'_i\|_{L^p} \Big)^pdt\\
& = I_1 + I_2.
\end{align*}
By Lemma \ref{multiplier}, for $t 2^{2i} \leq 1$, we have $ M(t,i)
\lesssim.$   Since $\al < \frac2p$, we take $a \in {\bf R}$
satisfying $\al -\frac2p < a < 0$ and using H$\ddot{o}$lder
inequality, we have
\begin{align*}
I_1 & \lesssim \int_0^T \Big(\sum_{t2^{2i} \leq 1 }2^{-\frac{p}{p-1}ai} \Big)^{p-1} \sum_{t2^{2i} \leq 1}
2^{pai} \| f * \phi'_i\|^p_{L^p}dt \\
& \lesssim  \int_0^T t^{\frac12pa }
\sum_{t2^{2i} \leq 1}
2^{pai} \| f * \phi'_i\|^p_{L^p}dt\\
& \lesssim \sum_{1 \leq i < \infty}
2^{pai} \| f * \phi'_i\|^p_{L^p} \int_0^{2^{-2i}}
t^{\frac12pa  }dt\\
& = c \sum_{1\leq  i <  \infty}
2^{-2i } \| f * \phi'_i\|^p_{L^p}.
\end{align*}
Now, we estimate $I_2$. By Lemma \ref{multiplier}, we have that $
M(t,i) \lesssim(t2^{2i})^{-m} \sum_{0 \leq i \leq L} t^i 2^{2ii}
\lesssim2^{(2L-2m)i} t^{L-m} $ for $t 2^{2i} \geq 1$ and $m>0$.  Let
us take $m$ and $b$ satisfying $b
>0$ and $\frac{p}2(2L -2m)
+ \frac12 pb +1 < 0$. Then, we get
\begin{align*}
&I_2   \lesssim \int_0^T   \Big(\sum_{t2^{2i} \geq 1}  2^{(2L-2m)i} t^{L-m} \| f *
\phi'_i\|_{L^p} \Big)^pdt \\
& \lesssim \int_0^\infty t^{\frac{p}2(2L -2m)  }
\Big(\sum_{t2^{2i} \geq 1}2^{-\frac{p}{p-1}
bi} \Big)^{p-1} \sum_{t2^{2i} \geq 1}
2^{pbi}2^{p(2L-2m)i} \| f * \phi'_i\|^p_{L^p}dt \\
& \lesssim \int_0^\infty t^{\frac{p}2( 2L-2m )  +
\frac12pb}
\sum_{t2^{2i} \geq 1}
2^{pbi}2^{p(2L-2m)i} \| f * \phi'_i\|^p_{L^p}dt\\
& \lesssim \sum_{1 \leq i < \infty}
2^{pbi}2^{p(2L-2m)i} \| f * \phi'_i\|^p_{L^p}
\int_{2^{-2i}}^\infty
t^{\frac{p}2(2L -2m )  + \frac12pb}dt\\
& =c\sum_{1 \leq i < \infty}
2^{-2i } \| f * \phi'_i\|^p_{L^p}.
\end{align*}
Hence, we complete the proof of theorem \ref{frac3p}.
\end{proof}

\begin{theo}\label{frac2p}
Let $1 \leq  p \leq  \infty$ and $i$ be a non-negative integer.
Let $ f \in {\mathcal B}^{2i - \frac2p}(\R)$ and $u$ be defined by
(\ref{main4}). Then, for $ T
> 0$, we have
\begin{align}\label{boundary2}
\| u\|_{W^{2i,i}_p({\bf R}^{n}_T)}\lesssim \| f\|_{{\mathcal
B}^{2i-\frac2p}_p (\R)}.
\end{align}
\end{theo}
\begin{proof}
From  Theorem \ref{frac3p},   (\ref{boundary2}) holds for $i =0$.

 Let $ i > 0$. We denote $\De= \sum_{1 \leq k\leq n} D^2_{X_k}$ and $ \De^{l+1} = \De \De^l$ for $l \geq 2$.
 Since $D^{l}_t D_X^\be u (X,t)= \De^lD_X^\be u(X,t) =
 < \De^l D_X^\be f, \Ga (X -\cdot,t)>$ for $|\be| + 2l \leq 2 i$, by Theorem \ref{frac3p},  we have
\begin{align*}
\|D^l_tD_X^\be u\|_{ L^p({\bf R}^{n}_T)}  & \lesssim \| \De^l
D_X^\be f\|_{{\mathcal B}_p^{-\frac2p} (\R)}
  \lesssim  \| f \|_{{\mathcal B}_p^{2i-\frac2p} (\R)}.
\end{align*}
For the last inequality, we used the well-known fact
\begin{align}\label{equal}
\|f\|_{{\mathcal B}^\al_p(\R)} \approx  \|f \|_{{\mathcal B}^{\al
-1}_p (\R)} + \| D_X f\|_{{\mathcal B}^{\al -1}_p (\R)}
\end{align}
for each $\al \in {\bf R}$ and $1 \leq p \leq \infty$ (see
\cite{BL}).
This completes the proof of Theorem \ref{frac2p}.
\end{proof}
In fact, for $ i \geq 1, \,\, 1 < p < \infty$ the Theorem \ref{frac2p} is known result
before (see \cite{La}).
\begin{theo}\label{inequality6}
Let $f \in {\mathcal
B}^{ -\frac2p}_p(\R)$ and $u$ be defined by (\ref{main4}).Then, for $ 1 \leq p \leq \infty$,
\begin{align}\label{eequivalent}
\| f\|_{{\mathcal B}^{-\frac2p}_p(\R)} & \lesssim \| u\|_{L^p (\R_T)}.
\end{align}
\end{theo}
\begin{proof}
Since the proof  of the case   $p = \infty$ is similar, we only prove the case $1 \leq p < \infty$.
Note that the
$L^p(\R)$-multiplier norms of $\hat{\phi}' (2^{-i} \xi) e^{|2^{-i}
\xi|^2}$ are equal to the $L^p(\R)$-multiplier norm of
 $\hat{\phi}'(\xi)
e^{|\xi|^2}$, where $\hat{\phi}'$ is defined in \eqref{psi}  (see
Theorem 6.1.3 in \cite{BL}). Using Lemma 6.1.5 in
 \cite{BL}, we have the $L^p(\R)$-multiplier norm of $ \hat{\phi}'(\xi) e^{|\xi|^2}$ is finite.
 Hence, for $1 \leq p < \infty$, we have
\begin{align*}
(2^{  -\frac2p i}\| f* \phi'_i\|_{L^p (\R)})^p &=  2^{  -2 i } \int_{\R} | {\mathcal F}^{-1} ( \hat{\phi}'(2^{-i}\xi)
e^{2^{-2i}|\xi|^2} e^{-2^{-2i}|\xi|^2} \hat{f})|^p
dX \\
& \lesssim2^{  -2 i}  \int_{\R} |  u (X, 2^{ - 2i}) |^p dX  \\
& \lesssim \int_{2^{-2i}}^{2^{-2i+2}}    \int_{\R}
| u (X, 2^{ - 2i}) |^p dX dt \\
& \lesssim  \int_{2^{-2i}}^{2^{-2i+2}}    \int_{\R}  \big( 2^{-i(n+2)} \int_{ J_{2^{-i -1}   (X, 2^{-2i} )}}  u(Y,s) dYds     \big)^p  dXdt  \\
& \lesssim \int_{2^{-2i}}^{2^{-2i+2}}
 \int_{\R} | u (X, t) |^p dX dt,
\end{align*}
where $J_r (X,t) = \{ (Y,s) \in {\bf R}^{n+1} \, | \, |X-Y| < r
,\,\,   |t-s|^\frac12 < r \}$.
Hence, we have
\begin{align*}
\sum_{ 1 \leq i <\infty} (2^{-\frac2p i } \| f*
\phi'_i\|_{L^p (\R)})^p
 & \lesssim  \sum_{1 \leq i <\infty}  \int_{2^{-2i}}^{2^{-2i+2}}
   \int_{\R}  |  u (X, t) |^p dX dt\\
 & \lesssim   \int_0^1
  \int_{\R}  |  u (X, t) |^p dX dt\\
  & \lesssim \int_0^1
 \int_{\R} | u (X, t) |^p dX dt.
\end{align*}
Similarly, the $L^p(\R)$-multiplier of $\hat{\psi}'(\xi) e^{\frac12
|\xi|^2}$ is finite. Hence,
we have
\begin{align*}
 \| f * \psi'\|^p_{L^p(\R) } & \lesssim  \int_{\R} | {\mathcal
F}^{-1} ( \psi' e^{ \frac12 |\xi|^2} e^{- \frac12|\xi|^2}
\hat{f})|^p \\
 & \lesssim \int_{\R}  | u (X, \frac12)|^p dX \\
 & \lesssim \int_{\R}  \int_{J_{\frac{1}4}
(X,\frac12)} |u(Y,s)|^p dYds  dX \\
 & \lesssim \int_0^1  \int_{\R} |u (Y,s)|^p dYds.
\end{align*}
Hence, we proved Theorem \ref{inequality6} when $T =1$.
For general $T > 0$,  we use scaling. Note that
\begin{align*}
v(X,t) = u(T^\frac12X,T t) = \int_{\R} \Ga(X-Y, t) f_T(Y) dY,
\end{align*}
where $f_T(Y) = f(T^\frac12Y)$. Hence, we have
\begin{align*}
 \|f_T\|_{ {\mathcal B}^{-\frac2p}_p  }   \lesssim\int_0^1\int_{\R} |v(X,t) |^p dXdt
 = c T^{-\frac{n+2}2} \int_0^T\int_{\R} |u(X,t) |^p dXdt.
\end{align*}
Since $\|f \|_{ {\mathcal B}^{-\frac2p}_p  }  \lesssim_T  \|f_T\|_{
{\mathcal B}^{-\frac2p}_p  } $, we obtain Theorem \ref{inequality6}
for general $0 < T < \infty$.




\end{proof}

\begin{theo}\label{inequality5}
Let $1 \leq p \leq \infty$ and $i$ be a non-negative integer. Let
$f \in {\mathcal B}^{2i-\frac2p}_p(\R)$ and $u$ is defined by
(\ref{main4}). Then
\begin{align}\label{equivalent}
\| f\|_{{\mathcal B}^{2i-\frac2p}_p(\R)} \lesssim \| u
\|_{W_p^{2i,i}({\bf R}^{n}_T)}.
\end{align}
\end{theo}

\begin{proof}
In Theorem \ref{inequality6},  we have
(\ref{equivalent}) for $ i=0$.
 Let $i > 0$. Notice that for $|\be| \leq 2i$,  we have
$D_{X}^\beta u (X,t) = c_n   \int_{\R} t^{-\frac{n}2}
e^{-\frac{|X-Y|^2}{4t}}  D_{Y}^\beta f (Y) dY. $  By \eqref{equal} and \eqref{eequivalent}, we have
\begin{align*}
\| f \|_{{\mathcal B}^{2i -\frac2p} (\R )} \lesssim \sum_{|\be|
\leq 2i} \| D_{X}^\beta f\|_{{\mathcal B}^{-\frac2p}_p (\R)}
\lesssim \sum_{|\be| + 2l \leq 2i} \|D_{X}^\beta u \|_{L^p ({\bf
R}^{n+1}_T)} \lesssim \| u\|_{W_p^{2i,i} ({\bf R}^{n}_T)} .
\end{align*}
This completes the proof of  Theorem \ref{inequality5}.
\end{proof}

Combining Theorem \ref{frac3p}-Theorem \ref{inequality5} and by
the real interpolation property, we obtain the result of Theorem
\ref{mainresult}.

\end{document}